\documentclass[11pt]{amsart}

\usepackage{amsmath,amssymb,amsthm}
\usepackage{microtype}
\usepackage[colorlinks=true,
  linkcolor=black,
  citecolor=black,
  urlcolor=black]{hyperref}

\newtheorem{theorem}{Theorem}
\newtheorem{proposition}[theorem]{Proposition}

\newtheorem*{theoremA}{Theorem A}

\theoremstyle{remark}

\theoremstyle{definition}
\newtheorem*{question*}{Question}
\newtheorem{question}{Question}

\newcommand{\Cr}{C^*_r}
\newcommand{\K}{\mathcal K}
\newcommand{\Cu}{\operatorname{Cu}}

\title[Pureness in reduced group $C^*$-algebras]
{Nonamenable groups whose reduced group $C^*$-algebras are not pure}

\author{Jamie Bell}

\address[Jamie Bell]{Mathematical Institute, University of M\"unster, Einsteinstr. 62, 48149 Munster, Germany}

\email{jbell@uni-muenster.de}

\thanks{Funded by the Deutsche Forschungsgemeinschaft
(DFG, German Research Foundation) under Germany's Excellence Strategy EXC 2044/2 -390685587, Mathematics M\"unster: Dynamics--Geometry--Structure and project-ID 427320536, SFB 1442, of the DFG}



\begin{document}

\begin{abstract}
We exhibit nonamenable groups whose reduced group $C^*$-algebras are not pure. More precisely, if $\Gamma$ is any countably infinite discrete group, then the reduced group $C^*$-algebra of the restricted wreath product $(\mathbb Z/2\mathbb Z)\wr\Gamma$ has an ideal-quotient isomorphic to $\K(\ell^2(\Gamma))$. It therefore fails to be nowhere scattered and, in particular, is not pure. Taking $\Gamma$ to be any nonamenable group, we thereby obtain a negative answer to a question of Thiel.
\end{abstract}

\maketitle

\section{Introduction}

A $C^*$-algebra $A$ is called \emph{pure} if its Cuntz semigroup is almost unperforated and almost divisible, equivalently, $\Cu(A)\cong \Cu(A)\otimes\Cu(\mathcal Z)$,
where $\mathcal Z$ denotes the Jiang--Su algebra; see
\cite{APTV,APT}. Originally introduced by Winter in his seminal work \cite{Winter}, pure $C^*$-algebras are now known to form a robust class, being closed under ideals, quotients, extensions, and inductive limits \cite{APTV,PTV}. Recall that an \emph{ideal-quotient} of a $C^*$-algebra $A$ is a $C^*$-algebra of the form $I/J$, where $J\subseteq I$ are ideals of $A$. A
$C^*$-algebra is \emph{nowhere scattered} if none of its nonzero ideal-quotients is elementary \cite{ThielVilalta}. Since pure $C^*$-algebras are nowhere scattered, the existence of an elementary
ideal-quotient is an obstruction to being pure.

If a discrete group $G$ is amenable, then its trivial representation factors through $\Cr(G)$ and yields a quotient isomorphic to $\mathbb C$. Since $\mathbb C$ is not pure, it follows that pureness of $\Cr(G)$ implies nonamenability of $G$. This raises the converse question.

\begin{question*}
Is $\Cr(G)$ pure whenever $G$ is nonamenable?
\end{question*}

This question first appeared explicitly on Thiel's list of open problems \cite{ThielWeb}, and was subsequently reiterated in \cite[Question~2]{Bell}. Here, we resolve the question negatively. Somewhat surprisingly, the solution comes from a very natural construction. Our main result is the following.

\begin{theoremA}\label{T-A}
Let $\Gamma$ be a countably infinite discrete group. Then the reduced group $C^*$-algebra $\Cr((\mathbb Z/2\mathbb Z)\wr\Gamma)$ has an ideal-quotient isomorphic to $\K(\ell^2(\Gamma))$ and is therefore not nowhere scattered and hence not pure.
\end{theoremA}

In previous work of the author \cite{Bell}, it was shown that reduced crossed products arising from equicontinuous actions of several classes of nonamenable groups are pure, and in the minimal case are often selfless. By contrast, the proof of Theorem~A involves identifying $\Cr((\mathbb Z/2\mathbb Z)\wr\Gamma)$ with the reduced crossed product $C(\{0,1\}^\Gamma)\rtimes_r \Gamma$ of the Bernoulli shift, which essentially lies at the opposite end of the dynamical spectrum. It therefore seems interesting to determine the regularity of crossed products of actions with intermediate dynamical properties (cf.~Question 2). 

Upon publication of this note, the author was informed that the idea of embedding the action of a group $\Gamma$ on its one-point compactification into $(\mathbb{Z}/2\mathbb{Z})\wr \Gamma$ was already used in the proof of \cite[Proposition 3.7]{BBLS}.

\medskip
\noindent\textit{Acknowledgements}. 
This work was completed during a visit to Chalmers University of Technology and the University of Gothenburg. I am grateful to the organisers of the workshop \textit{Cuntz Semigroups and Dynamics} for their hospitality during this time. I also thank Eduardo Scarparo for bringing \cite{BBLS} to my attention. 

\medskip
\noindent\textit{AI Declaration}.
ChatGPT-5.6 Sol was used during the preparation of this work, including identifying the wreath-product example and preparing a draft of the manuscript. The author subsequently verified the mathematical arguments and takes full responsibility for the contents of the paper.

\section{Compactified free orbits}

Let $\Gamma$ be a countably infinite discrete group, and let $\Gamma^+=\Gamma\sqcup\{\infty\}$ be its one-point compactification. We let $\Gamma$ act on $\Gamma^+$
by left translation on $\Gamma$ and fixing $\infty$.

\begin{proposition}\label{prop:orbit}
Let $\Gamma\curvearrowright X$ be an action on a compact Hausdorff space. Suppose there is a closed invariant subspace $Y\subseteq X$ which is $\Gamma$-equivariantly homeomorphic to $\Gamma^+$. Then $C(X)\rtimes_r\Gamma$ has an ideal-quotient isomorphic to $\K(\ell^2(\Gamma))$.
\end{proposition}

\begin{proof}
Since $Y$ is closed and invariant, $C(Y)$ is an equivariant quotient of $C(X)$, hence the restriction map induces a surjection $C(X)\rtimes_r\Gamma\twoheadrightarrow C(Y)\rtimes_r\Gamma$. The open invariant subset $Y\setminus\{\infty\}$ is equivariantly homeomorphic to $\Gamma$, by assumption. Consequently, $C_0(\Gamma)\rtimes_r\Gamma$ embeds as an ideal in $C(Y)\rtimes_r\Gamma$. One has the standard
isomorphism $C_0(\Gamma)\rtimes_r\Gamma\cong\K(\ell^2(\Gamma))$, 
as a special case of Green's imprimitivity theorem. Thus $\K(\ell^2(\Gamma))$ is an ideal-quotient of $C(X)\rtimes_r\Gamma$.
\end{proof}

\section{The wreath product counterexample}

Consider the (restricted) wreath product $G=(\mathbb Z/2\mathbb Z)\wr\Gamma = (\bigoplus_\Gamma\mathbb Z/2\mathbb Z)\rtimes\Gamma$, where $\Gamma$ acts by shifting the coordinates. We have
\begin{equation}\label{eq:bernoulli}
   \Cr(G)\cong C(\{0,1\}^{\Gamma})\rtimes_r\Gamma,
\end{equation}
where the action on $\{0,1\}^{\Gamma}$ is the Bernoulli shift. Let
$x_\infty$ denote the zero configuration in $\{0,1\}^{\Gamma}$. For
each $\gamma\in\Gamma$, let $x_\gamma$ be the configuration supported
at the single coordinate $\gamma$, and set
\[
   Y=\{x_\infty\}\cup\{x_\gamma:\gamma\in\Gamma\}
    =\bigl\{x\in\{0,1\}^{\Gamma}:
       |\operatorname{supp}(x)|\leq 1\bigr\},
\]
where $\operatorname{supp}(x)
   =\{\gamma\in\Gamma:x(\gamma)=1\}$. The subset $Y$ is invariant under the Bernoulli shift. It is also
closed. Indeed, if $x\notin Y$, then there exist distinct
$\gamma,\eta\in\Gamma$ such that $x(\gamma)=x(\eta)=1$, and the
cylinder set
\[
   U_{\gamma,\eta}
   =\{z\in\{0,1\}^{\Gamma}:z(\gamma)=z(\eta)=1\}
\]
is an open neighbourhood of $x$ disjoint from $Y$. Hence
$\{0,1\}^{\Gamma}\setminus Y$ is open. The bijection $\varphi : \Gamma^+\to Y$ given by $\gamma\mapsto x_\gamma$ and $\infty\mapsto x_\infty$ is an equivariant homeomorphism. Indeed,
every $x_\gamma$ is isolated in $Y$, while a neighbourhood basis at
$x_\infty$ is given by the sets $\{x_\infty\}\cup
   \{x_\gamma:\gamma\in\Gamma\setminus F\}$, for $F\subseteq\Gamma$ finite, which correspond under $\varphi$ to the standard neighbourhood basis at $\infty$ in the one-point compactification $\Gamma^+$. Thus $Y\cong\Gamma^+$ equivariantly.

\begin{proof}[Proof of Theorem A]
The existence of the ideal-quotient follows from
\eqref{eq:bernoulli}, Proposition~\ref{prop:orbit}, and the closed invariant subspace $Y$ constructed above. It remains to verify the assertion about pureness. If the reduced group $C^*$-algebra were pure, then so would be each of its ideal-quotients, since pureness passes to ideals and quotients. On the other hand, $\Cu(\K)\cong\{0,1,2,\ldots,\infty\}$ is not almost divisible. For example, there is no $y$ such that $2y\leq 1\leq 3y$. Thus $\K(\ell^2(\Gamma))$ is not pure, giving a contradiction.
\end{proof}

We conclude by posing the following natural questions.

\begin{question}
Is there a group-theoretic characterisation of those discrete groups $G$ for which $\Cr(G)$ is pure?
\end{question}

%In light of the results here and \cite{Bell}, we also ask the following. 

\begin{question}
Let $G\curvearrowright X$ be an action of a (nonamenable) discrete group on a compact Hausdorff space. Under which dynamical conditions is $C(X)\rtimes_r G$ nowhere scattered?
\end{question}


\begin{thebibliography}{99}

\bibitem{APT}
R.~Antoine, F.~Perera, and H.~Thiel.
Tensor products and regularity properties of Cuntz semigroups.
{\it Mem. Amer. Math. Soc.} \textbf{251} (2018).

\bibitem{APTV}
R.~Antoine, F.~Perera, H.~Thiel, and E.~Vilalta.
Pure $C^*$-algebras.
arXiv:2406.11052.

\bibitem{Bell}
J.~Bell.
Nonnuclear Bunce--Deddens algebras.
arXiv:2607.28597.

\bibitem{BBLS}
K.~A.~Brix, C.~Bruce, K.~Li, and E.~Scarparo. 
Maximal ideals of reduced group $C^*$-algebras and Thompson's groups.
{\it Trans. Amer. Math. Soc.} {\bf 379} (2026), 4307--4321.

\bibitem{PTV}
F.~Perera, H.~Thiel, and E.~Vilalta.
Extensions of pure $C^*$-algebras.
{\it Trans. Amer. Math. Soc.} (to appear).

\bibitem{ThielVilalta}
H.~Thiel and E.~Vilalta.
Nowhere scattered $C^*$-algebras.
{\it J. Noncommut. Geom.} \textbf{18} (2024), 231--263.

\bibitem{ThielWeb}
H.~Thiel.
OA-Problems: Pure group $C^*$-algebras.
\url{https://hannesthiel.org/pure-group-c-algebras}.

\bibitem{Winter}
W.~Winter.
Nuclear dimension and $\mathcal Z$-stability of pure $C^*$-algebras.
{\it Invent. Math.} \textbf{187} (2012), 259--342.

\end{thebibliography}
\end{document}